\documentclass[12pt, reqno]{amsart}
\usepackage{amsmath, amsthm, amscd, amsfonts, amssymb, graphicx, color}
\usepackage[bookmarksnumbered, colorlinks, plainpages]{hyperref}

\newtheorem{theorem}{Theorem}[section]
\newtheorem{lemma}[theorem]{Lemma}

\newtheorem{corollary}[theorem]{Corollary}
\theoremstyle{definition}

\newtheorem{example}[theorem]{Example}

\theoremstyle{remark}
\newtheorem{remark}[theorem]{Remark}
\numberwithin{equation}{section}
\input{mathrsfs.sty}
\begin{document}

\title{Non-commutative $f$-divergence functional}

\author[M.S. Moslehian, M. Kian]{Mohammad Sal Moslehian and Mohsen Kian}

\address{Department of Pure Mathematics, Center of Excellence in Analysis on Algebraic Structures (CEAAS), Ferdowsi University of Mashhad, P.O. Box 1159, Mashhad 91775, Iran.}
\email{moslehian@um.ac.ir and moslehian@member.ams.org}
\email{kian$_{-}$tak@yahoo.com}



\begin{abstract}
We introduce the non-commutative $f$-divergence functional
$\Theta(\widetilde{A},\widetilde{B}):=\int_TB_t^{\frac{1}{2}}f\left(B_t^{-\frac{1}{2}}
A_tB_t^{-\frac{1}{2}}\right)B_t^{\frac{1}{2}}d\mu(t)$ for an
operator convex function $f$, where $\widetilde{A}=(A_t)_{t\in T}$
and $\widetilde{B}=(B_t)_{t\in T}$ are continuous fields of Hilbert
space operators and study its properties. We establish some
relations between the perspective of an operator convex function $f$
and the non-commutative $f$-divergence functional. In particular, an
operator extension of Csisz\'{a}r's result regarding $f$-divergence
functional is presented. As some applications, we establish a
refinement of the Choi--Davis--Jensen operator inequality, obtain
some unitarily invariant norm inequalities and give some results
related to the Kullback--Leibler distance.
\end{abstract}

\maketitle

\section{Introduction and Preliminaries}
\noindent

Let $\mathbb{B}(\mathscr{H})$ be the algebra of all bounded linear
operators on a complex Hilbert space $\mathscr{H}$ and $I$ denote
the identity operator. If $\dim\mathscr{H}=n$, we identify
$\mathbb{B}(\mathscr{H})$ with the algebra
$\mathcal{M}_n(\mathbb{C})$ of all $n\times n$ matrices with entries
in the complex number field $\mathbb{C}$. An operator $A$ is said to
be positive (denoted by $A\geq 0$) if $\langle Ax,x\rangle\geq 0$
for all vectors $x\in\mathscr{H}$. If, in addition, $A$ is
invertible, then it is called strictly positive (denoted by $A> 0$).
By $A\geq B$ we mean that $A-B$ is positive, while $A>B$ means that
$A-B$ is strictly positive. A map $\Phi$ on
$\mathbb{B}(\mathscr{H})$ is called positive if $\Phi(A)\geq0$ for
each $A\geq0$. An operator $C$ is called an isometry if $C^*C=I$, a
contraction if $C^*C\leq I$ and an expansive operator if $C^*C\geq
I$.

A continuous real valued function $f$ defined on an interval $J$ is
said to be operator convex if
\begin{eqnarray*}
 f(\lambda A+(1-\lambda)B)\leq \lambda f(A)+(1-\lambda)f(B),
\end{eqnarray*}
for all self-adjoint operators $A,B$ with spectra contained in $J$
and any $\lambda\in[0,1]$. If $-f$ is operator convex, then $f$ is
said to be operator concave. Let $J_1$ and $J_2$ be two real
intervals. A jointly operator convex function is a function $f$
defined on $J_1\times J_2$ such that
 \begin{eqnarray*}
 f(\lambda (A,B)+(1-\lambda)(C,D))\leq \lambda f(A,B)+(1-\lambda)f(C,D),
 \end{eqnarray*}
 for all self-adjoint operators $A,C$ with spectra contained in $J_1$, all self-adjoint operators $B,D$ with spectra contained in $J_2$ and all $\lambda\in[0,1]$; see e.g. \cite{MA} for the definition of $f(A,B)$.

The Jensen operator inequality, due to Hansen and Pedersen states
that $f:J\to\mathbb{R}$ is operator convex if and only if
\begin{eqnarray}\label{hpj}
 f(C^*AC)\leq C^*f(A)C,
\end{eqnarray}
for any isometry $C$ and any self-adjoint operator $A$ with spectrum
contained in $J$, see \cite{Fu} for various equivalent assertions.
If $0 \in J$ and $f(0)\leq 0$, then $f$ is operator convex on $J$ if
and only if \eqref{hpj} holds for any contraction $C$. Some other
various characterizations of operator convexity can be found in
\cite[Chapter 1]{Fu}; see also \cite{MOS2, MOS1} and references
therein.

The Choi--Davis--Jensen inequality states that if $f$ is operator
convex, then
\begin{eqnarray}\label{choi}
 f(\Phi(A))\leq\Phi(f(A)),
\end{eqnarray}
for any unital positive linear map $\Phi$ and any self-adjoint operator $A$, whose spectrum is contained in the domain of $f$; see \cite{PET} for a characterization for the case of equality. An extension of this significant inequality reads as follows.\\
\textbf{Theorem A.} \cite{Mo} \textit{Let $f$ be an operator convex
function on $J$ and $\Phi_1,\cdots,\Phi_n$ be positive linear maps
on $\mathbb{B}(\mathscr{H})$ with $\sum_{i=1}^{n}\Phi_i(I)=I$. Then
\begin{eqnarray}\label{jo}
 f\left(\sum_{i=1}^{n}\Phi_i(A_i)\right)\leq \sum_{i=1}^{n}\Phi_i(f(A_i)),
\end{eqnarray}
for all self-adjoint operators $A_i$ \ $(i=1,\cdots,n)$ with spectra
contained in $J$.} In particular,
\begin{eqnarray}\label{multij}
 f\left(\sum_{i=1}^{n}C_i^*A_iC_i\right)\leq\sum_{i=1}^{n}C_i^*f(A_i)C_i,
\end{eqnarray}
whenever $\sum_{i=1}^{n}C_i^*C_i=I$.

Let $T$ be a locally compact Hausdorff space and $\mathfrak{A}$ be a
$C^*$-algebra of Hilbert space operators. A field $(A_t)_{t\in T}$
of operators in $\mathfrak{A}$ is said to be continuous if the
function $t\mapsto A_t$ is norm continuous on $T$. Moreover, If
$\mu$ is a Radon measure on $T$ and the function $t\mapsto A_t$ is
integrable on $T$, then the Bochner integral $\int_TA_td\mu(t)$ is
defined to be the unique element of $\mathfrak{A}$ with the property
that
\begin{eqnarray*}
\rho\left(\int_TA_td\mu(t)\right)=\int_T\rho(A_t)d\mu(t),
\end{eqnarray*}
for any linear functional $\rho$ in the norm dual $\mathfrak{A}^*$
of $\mathfrak{A}$.

Furthermore, let $\mathfrak{A}$ and $\mathfrak{B}$ be $C^*$-algebras
of operators. A field $(\Phi_t)_{t\in T}:\mathfrak{A}\to
\mathfrak{B}$ of positive linear maps is said to be continuous if
the function $t\mapsto\Phi_t(A)$ is continuous on $T$ for every
$A\in\mathfrak{A}$. If the $C^*$-algebras $\mathfrak{A}$ and
$\mathfrak{B}$ are unital and the function $t\mapsto \Phi_t(I)$ is
integrable on $T$ with integral $I$, then we say that the field
$(\Phi_t)_{t\in T}$ is unital.

The following result, is the Jensen operator inequality for continuous fields of operators.\\
\textbf{Theorem B.} \cite{HPP} \textit{Let $f$ be an operator convex
function defined on an interval $J$, and let $\mathfrak{A}$ and
$\mathfrak{B}$ be unital $C^*$-algebras. If $(\Phi_t)_{t\in
T}:\mathfrak{A}\to \mathfrak{B}$ is a unital field of positive
linear maps defined on a locally compact Hausdorff space $T$ with a
bounded Radon measure $\mu$, then
\begin{eqnarray}\label{ji}
 f\left(\int_T\Phi_t(A_t)d\mu(t)\right)\leq\int_T\Phi_t(f(A_t))d\mu(t),
\end{eqnarray}
for every norm bounded continuous field $(A_t)_{t\in T}$ of
self-adjoint operators in $\mathfrak{A}$ with spectra contained in
$J$.}

Let $f$ be a convex function on a convex set $K\subseteq\mathbb{R}$.
Following \cite{H}, the perspective function $g$ associated to $f$
is defined on the set $\{(x,y):y>0 \quad
\mbox{and}\quad\frac{x}{y}\in K\}$ by
\begin{eqnarray*}
 g(x,y):=yf\left(\frac{x}{y}\right).
\end{eqnarray*}

As an operator extension of the perspective function, Effros
\cite{E} introduced the perspective function of an operator convex
function $f$ by
\begin{eqnarray*}
 g(L,R):=Rf\left(\frac{L}{R}\right),
\end{eqnarray*}
for commuting strictly positive operators $L$ and $R$ and proved the following notable theorem.\\
\textbf{Theorem C.} \cite{E} \label{Ef1}
 \textit{If $f$ is operator convex, when restricted to the commuting strictly positive operators, then the perspective function $(L,R)\mapsto g(L,R)=Rf\left(\frac{L}{R}\right)$ is jointly operator convex.}

 He also extended the generalized perspective function, defined by Mar\'echal \cite{M,M2} to operators.
Given continuous functions $f$ and $h$ and commuting strictly
positive operators $L$ and $R$, Effros defined the operator
extension of the generalized perspective function by
\begin{eqnarray*}
 (f\Delta h)(L,R):=h(R)f\left(\frac{L}{h(R)}\right),
\end{eqnarray*}
and proved the following assertion.\\
\textbf{Theorem D.} \label{Ef2} \textit{If $f$ is operator convex
with $f(0)\leq0$ and $h$ is operator concave with $h>0$, then
$f\Delta h$ is jointly convex on commuting strictly positive
operators.}

The authors of \cite{Eb} extended Effros results by removing the
restriction to commuting operators and proved analogue results to
Theorem C and Theorem D.

An interesting study of such functions for operators was introduced
by Kubo and Ando. They considered the case where $f$ is an operator
monotone function and established a relation between the operator
monotone functions and the operator means (see \cite[Chapter
5]{Fu}).

One of the most principal matters in applications of probability
theory is to find a suitable measure between two probability
distributions. Many kinds of such measures have been studied and
applied in several fields such as signal processing, genetics and
economics. One of the most famous of such measures is the
Csisz\'{a}r $f$-divergence functional, which includes several
measures.

For a convex function $f:[0,\infty)\to\mathbb{R}$, Csisz\'{a}r
\cite{C,C2} introduced the $f$-divergence functional by
\begin{eqnarray}\label{def22}
 I_f(\widetilde{p},\widetilde{q}):=\sum_{i=1}^{n}q_if\left(\frac{p_i}{q_i}\right),
\end{eqnarray}
for positive $n$-tuples $\widetilde{p}=(p_1,\cdots,p_n)$ and
$\widetilde{q}=(q_1,\cdots,q_n)$, in which undefined expressions
were interpreted by
\begin{eqnarray*}
 f(0)=\lim_{t\to 0^+}f(t), \qquad 0f\left(\frac{0}{0}\right)=0, \qquad
 0f\left(\frac{p}{0}\right)=\lim_{\epsilon\to 0^+}f\left(\frac{p}{\epsilon}\right)=p\lim_{t\to \infty}\frac{f(t)}{t}.
 \end{eqnarray*}
Also Csisz\'{a}r and K\"{o}rner \cite{CK} obtained the following result.\\

\textbf{Theorem E.} \textit{If $f:[0,\infty)\to\mathbb{R}$ is
convex, then $I_f(\widetilde{p},\widetilde{q})$ is jointly convex in
$\widetilde{p}$ and $\widetilde{q}$ and
 \begin{eqnarray}\label{m1}
 \sum_{i=1}^{n}q_if\left(\frac{\sum_{i=1}^{n}p_i}{\sum_{i=1}^{n}q_i}\right)\leq I_f(\widetilde{p},\widetilde{q})
 \end{eqnarray}
 for all positive $n$-tuples $\widetilde{p}=(p_1,\cdots,p_n), \widetilde{q}=(q_1,\cdots,q_n)$.}

A series of results and inequalities related to $f$-divergence
functionals can be found in \cite{A,CD, DK,H2}.

The paper is organized as follows. In section 2, we introduce an
operator extension of $f$-divergence functional for an operator
convex function $f$, which is more general than the perspective
function associated to $f$. We give some properties of our
non-commutative $f$-divergence functional and establish its
relationship to the perspective of $f$. In particular, an operator
extension of \eqref{m1} is presented. In section 3, we provide some
applications for our results. More precisely, a refinement of the
Choi--Davis--Jensen operator inequality is obtained, some unitarily
invariant norm inequalities are presented and some results related
to the Kullback--Leibler distance are given.

\section{ Non-commutative $f$-divergence functionals}
Throughout this section, assume that $T$ is a locally compact
Hausdorff space with a bounded Radon measure $\mu$ and
$\mathfrak{A}$ and $\mathfrak{B}$ are $C^*$-algebras of Hilbert
space operators. Assume that $\widetilde{A}=(A_t)_{t\in T}$ and
$\widetilde{B}=(B_t)_{t\in T}$ are continuous fields of self-adjoint
and strictly positive operators in $\mathfrak{A}$, respectively, and
$(\Phi_t)_{t\in T}:\mathfrak{A}\to\mathfrak{B}$ is a unital field of
positive linear maps.
 Furthermore, when $T$ is the finite set $\{1,\cdots,n\}$ and $\mu$ is the counting measure, suppose that $\widetilde{L}=(L_1,\cdots,L_n)$ and $\widetilde{R}=(R_1,\cdots,R_n)$ are $n$-tuples of self-adjoint and strictly positive operators on a Hilbert space
 $\mathscr{H}$, respectively, and $(\Phi_1,\cdots,\Phi_n)$ is an $n$-tuple of positive linear maps on $\mathbb{B}(\mathscr{H})$.

 Let $f:[0,\infty)\to\mathbb{R}$ be a convex function. The perspective function $g$ associated to $f$ is defined by
\begin{eqnarray*}
 g(L,R):=R^{\frac{1}{2}}f(R^{-\frac{1}{2}}LR^{-\frac{1}{2}})R^{\frac{1}{2}},
\end{eqnarray*}
 where $L$ is a self-adjoint operator and $R$ is a strictly positive operator on a Hilbert space $\mathscr{H}$.
In \cite{Eb} it is shown that $f$ is operator convex if and only if
$g$ is jointly operator convex. We consider a more general case. Let
us define the non-commutative $f$-divergence functional $\Theta$ by
\begin{eqnarray}\label{def1}
 \Theta(\widetilde{A},\widetilde{B}):=\int_TB_t^{\frac{1}{2}}f\left(B_t^{-\frac{1}{2}}
 A_tB_t^{-\frac{1}{2}}\right)B_t^{\frac{1}{2}}d\mu(t).
\end{eqnarray}
 Hence, in the discrete case $\Theta$ is defined by
 \begin{eqnarray}\label{def222}
 \Theta(\widetilde{L},\widetilde{R})=\sum_{i=1}^{n}R_i^{\frac{1}{2}}f(R_i^{-\frac{1}{2}}L_iR_i^{-\frac{1}{2}})R_i^{\frac{1}{2}}.
\end{eqnarray}
 By the same argument as in \cite{Eb}, it is easy to see that $\Theta$ is jointly operator convex if and only if $f$ is operator convex.
In the sequel, we study some properties of $\Theta$ and establish
some relations between $\Theta$ and $g$. The following result is an
extension of \eqref{m1}.

\begin{theorem}\label{th2}
 Let $f$ be an operator convex function, and $g$ be the corresponding perspective function. Then
 \begin{eqnarray}\label{q2}
 g(A,B)\leq\Theta(\widetilde{A},\widetilde{B}),
 \end{eqnarray}
 where $A=\int_TA_td\mu(t)$ and $B=\int_TB_td\mu(t)$.
\end{theorem}
\begin{proof}
\begin{align*}
 f&\left(B^{-\frac{1}{2}}AB^{-\frac{1}{2}}\right)\\
 & =
 f\left(\left(\int_TB_sd\mu(s)\right)^{-\frac{1}{2}}\int_TA_td\mu(t)
\left(\int_TB_sd\mu(s)\right)^{-\frac{1}{2}}\right)\\
&=
 f\left(\int_T\left(\int_TB_sd\mu(s)\right)^{-\frac{1}{2}}A_t
\left(\int_TB_sd\mu(s)\right)^{-\frac{1}{2}}d\mu(t)\right)\\
&=
f\left(\int_T\left(\int_TB_sd\mu(s)\right)^{-\frac{1}{2}}B_t^{\frac{1}{2}}
\left(B_t^{-\frac{1}{2}}A_tB_t^{-\frac{1}{2}}\right)B_t^{\frac{1}{2}}
\left(\int_TB_sd\mu(s)\right)^{-\frac{1}{2}}d\mu(t)\right)\\
&\leq
\int_T\left(\int_TB_sd\mu(s)\right)^{-\frac{1}{2}}B_t^{\frac{1}{2}}
f\left(B_t^{-\frac{1}{2}}A_tB_t^{-\frac{1}{2}}\right)B_t^{\frac{1}{2}}
\left(\int_TB_sd\mu(s)\right)^{-\frac{1}{2}}d\mu(t)\\
& \qquad\qquad\qquad\qquad\qquad\qquad (\mbox{ by the Jensen operator inequality \eqref{ji}}) \\
&=\left(\int_TB_sd\mu(s)\right)^{-\frac{1}{2}}\int_TB_t^{\frac{1}{2}}
f\left(B_t^{-\frac{1}{2}}A_tB_t^{-\frac{1}{2}}\right)B_t^{\frac{1}{2}}d\mu(t)
\left(\int_TB_sd\mu(s)\right)^{-\frac{1}{2}}\\
&=
B^{-\frac{1}{2}}\Theta(\widetilde{A},\widetilde{B})B^{-\frac{1}{2}},
\end{align*}
whence we obtain the desired inequality \eqref{q2}.
\end{proof}
\begin{corollary}\label{qw1}
Let $f$ be an operator convex function and $g$ be the perspective
function of $f$. Then
\begin{enumerate}
\item[(i)] The perspective function $g$ of an operator convex
function $f$ is sub-additive. More general,
\begin{eqnarray}\label{jd1}
g\left(\sum_{i=1}^{n}L_i,\sum_{i=1}^{n}R_i\right)\leq\sum_{i=1}^{n}g(L_i,R_i)\,.
\end{eqnarray}
\item[(ii)] $f\left(\sum_{i=1}^{n}L_i\right)\leq\sum_{i=1}^{n}g(L_i,R_i)$, whenever $\sum_{i=1}^{n}R_i=I$.
\end{enumerate}
\end{corollary}

Let $T_1$ and $T_2$ be disjoint locally compact Hausdorff spaces and
$T=T_1\cup T_2$. The following refinement of \eqref{q2} holds.
\begin{corollary}\label{cor3}
 Let $g$ be the perspective function of an operator convex function $f$. Then
 \begin{eqnarray}\label{q1}
 2g\left(\frac{1}{2}(A,B)\right)\leq g(A_{T_1},B_{T_1})+g(A_{T_2},B_{T_2})\leq\Theta(\widetilde{A},\widetilde{B}),
 \end{eqnarray}
 where $A=\int_TA_td\mu(t)$, $B=\int_TB_td\mu(t)$, $A_{T_1}=\int_{T_1}A_td\mu(t)$ and $B_{T_1}=\int_{T_1}B_td\mu(t)$.
 \end{corollary}
 \begin{proof}
 Since $(A,B)=(A_{T_1},B_{T_1})+(A_{T_2},B_{T_2})$, the first inequality of \eqref{q1} follows from the joint convexity of $g$. The second inequality follows immediately from Theorem \ref{th2}.
 \end{proof}

 \begin{theorem}\label{tht3}
 Let $L_{ij}$ and $R_{ij}$ \,\, $(1\leq i,j\leq n)$ be self-adjoint and strictly positive operators, respectively, and let
 $p_j$ \,\, $(1\leq j\leq n)$ be positive numbers. If $f$ is operator convex, then
 \begin{eqnarray*}
 \sum_{i=1}^{n}g(L_i,R_i)\leq\sum_{i=1}^{n}p_i\Theta(\widetilde{L}^i,\widetilde{R}^i),
 \end{eqnarray*}
 where $L_i=\sum_{j=1}^{n}p_jL_{ij}$, $R_i=\sum_{j=1}^{n}p_jR_{ij}$, $\widetilde{L}^i=(L_{i1},\cdots,L_{in})$, $\widetilde{R}^i=(R_{i1},\cdots,R_{in})$.
 \end{theorem}

 \begin{proof}
 Using \eqref{q2} for $A_i$ and $B_i$\,\, $(1\leq i\leq n)$ we obtain
 \begin{eqnarray}\label{q16}
 g(L_i,R_i)=R_i^{\frac{1}{2}} f\left(R_i^{-\frac{1}{2}}L_iR_i^{-\frac{1}{2}}\right)R_i^{\frac{1}{2}}\leq
 \Theta(p\widetilde{L}^i,p\widetilde{R}^i), \ \ (1\leq i\leq n),
 \end{eqnarray}
 where $p\widetilde{A}^i=(p_1A_{i1},\cdots,p_nA_{in})$ and $p\widetilde{B}^i=(p_1B_{i1},\cdots,p_nB_{in})$. In addition,
 \begin{eqnarray}\label{q7}
  \Theta(p\widetilde{L}^i,p\widetilde{R}^i)&=&\sum_{j=1}^{n}(p_jR_{ij})^{\frac{1}{2}}
  f\left((p_jR_{ij})^{-\frac{1}{2}}
  (p_jL_{ij})(p_jR_{ij})^{-\frac{1}{2}}\right)(p_jR_{ij})^{\frac{1}{2}}\nonumber\\
  &=& \sum_{j=1}^{n}p_jR_{ij}^{\frac{1}{2}}f(R_{ij}^{-\frac{1}{2}}
  L_{ij}R_{ij}^{-\frac{1}{2}})R_{ij}^{\frac{1}{2}}.
 \end{eqnarray}
 Summing \eqref{q16} over $i$ we get
 \begin{eqnarray*}
  \sum_{i=1}^{n}g(L_i,R_i)&\leq& \sum_{i=1}^{n}\Theta(p\widetilde{L}^i,p\widetilde{R}^i)\\
  &=& \sum_{i=1}^{n}\sum_{j=1}^{n}p_jR_{ij}^{\frac{1}{2}}f(R_{ij}^{-\frac{1}{2}}
  L_{ij}R_{ij}^{-\frac{1}{2}})R_{ij}^{\frac{1}{2}}\qquad\qquad\qquad(\mbox{by \eqref{q7}})\\
  &=& \sum_{j=1}^{n}p_j\sum_{i=1}^{n}R_{ij}^{\frac{1}{2}}f(R_{ij}^{-\frac{1}{2}}
  L_{ij}R_{ij}^{-\frac{1}{2}})R_{ij}^{\frac{1}{2}}\\
  &=& \sum_{j=1}^{n}p_j \Theta(\widetilde{L}^i,\widetilde{R}^i).
 \end{eqnarray*}
 \end{proof}
For continuous functions $f$ and $h$ and commuting matrices $L$ and
$R$, Effros \cite{E} defined the function $(L,R)\mapsto(f\Delta
h)(L,R)$ by
 \begin{eqnarray*}
 (f\Delta h)(L,R):=f\left(\frac{L}{h(R)}\right)h(R).
 \end{eqnarray*}
 He also proved that if $f$ is operator convex with $f(0)\leq0$ and $h$ is operator concave with $h>0$, then $f\Delta h$ is jointly operator convex. In \cite{Eb}, definition and properties of $f\Delta h$ were naturally given for two not necessarily commuting self-adjoint operators.

Let $f$ and $h$ be continuous real valued functions defined on an
interval $J$ and $\mu$ be a probability measure on $T$. As a
generalization of $f\Delta h$, we define $f\nabla h$ by
\begin{eqnarray*}
(f\nabla h)(\widetilde{A},\widetilde{B}):=\int_Th(B_t)^{\frac{1}{2}}
f\left(h(B_t)^{-\frac{1}{2}}A_t
h(B_t)^{-\frac{1}{2}}\right)h(B_t)^{\frac{1}{2}}d\mu(t).
\end{eqnarray*}
 It is not hard to see that $f$ is operator convex with $f(0)\leq0$ and $h$ is operator concave with $h>0$ if and only if $f\nabla h$ is jointly operator convex.

The next result, is a Choi--Davis--Jensen type inequality for
$f\Delta h$.
\begin{theorem}\label{th10}
 Let $f$ be an operator convex function with $f(0)\leq0$ and $h$ be an operator concave function with $h>0$. If $\int_T\Phi_t(I)d\mu(t)\leq I$, then
 \begin{eqnarray}\label{qqi1}
 (f\Delta h)\left(\int_T\Phi_t(A_t)d\mu(t),\int_T\Phi_t(B_t)d\mu(t)\right)\leq\int_T\Phi_t((f\Delta h)(A_t,B_t))d\mu(t).
 \end{eqnarray}
 In particular, if $g$ is the perspective function of $f$, then
 \begin{eqnarray}\label{qqqq1}
 g\left(\int_T\Phi_t(A_t)d\mu(t),\int_T\Phi_t(B_t)d\mu(t)\right)\leq\int_T\Phi_t(g(A_t,B_t))d\mu(t),
  \end{eqnarray}
  where $B_t$ is strictly positive for any $t\in T$.
\end{theorem}
\begin{proof}
Let $(B_t)_{t\in T}$ be a continuous filed of self-adjoint
operators. Define the field of positive linear maps
$\Psi_t:\mathfrak{A}\to\mathfrak{B}$ by
\begin{eqnarray*}
 \Psi_t(X)=h\left(\int_T\Phi_t(B_t)d\mu(t)\right)^{-\frac{1}{2}}\Phi_t
 \left(h(B_t)^{\frac{1}{2}}
 Xh(B_t)^{\frac{1}{2}}\right)h\left(\int_T\Phi_t(B_t)d\mu(t)\right)^{-\frac{1}{2}}.
 \end{eqnarray*}
 Since $h$ is operator concave, $h>0$ and $\int_T\Phi_t(I)d\mu(t)\leq I$, we have
 \begin{align*}
  \int_T\Phi_t(h(B_t))d\mu(t)\leq h\left(\int_T\Phi_t(B_t)d\mu(t)\right).
  \end{align*}
  Therefore
 \begin{eqnarray*}
 \int_T\Psi_t(I)d\mu(t)=\int_Th\left(\int_T\Phi_t(B_t)d\mu(t)\right)^{-\frac{1}{2}}\Phi_t(h(B_t))
 h\left(\int_T\Phi_t(B_t)d\mu(t)\right)^{-\frac{1}{2}}\leq
 I.
 \end{eqnarray*}
 Put $C=\int_T\Phi_t(B_t)d\mu(t)$. We have
 \begin{align*}
 (f\Delta h)&\left(\int_T\Phi_t(A_t)d\mu(t),\int_T\Phi_t(B_t)d\mu(t)\right)\\
 &=h(C)^{\frac{1}{2}}f\left(h
 (C)^{-\frac{1}{2}}\int_T\Phi_t(A_t)d\mu(t)\, h(C)^{-\frac{1}{2}}\right)
 h(C)^{\frac{1}{2}}\\
 &= h(C)^{\frac{1}{2}}f\left(\int_Th
 (C)^{-\frac{1}{2}}\, \Phi_t(A_t)\,
 h(C)^{-\frac{1}{2}} d\mu(t)\right)
 h(C)^{\frac{1}{2}}\\
 &= h(C)^{\frac{1}{2}}f\left(\int_T\Psi_t\left(h(B_t)^{-\frac{1}{2}}\, A_t\,
 h(B_t)^{-\frac{1}{2}} \right)d\mu(t)\right)
 h(C)^{\frac{1}{2}}\\
 & \leq h(C)^{\frac{1}{2}}\int_T\Psi_t\left(f\left(h(B_t)^{-\frac{1}{2}}\, A_t\,
 h(B_t)^{-\frac{1}{2}} \right)\right)d\mu(t)
 h(C)^{\frac{1}{2}}\\
 & \qquad\qquad\qquad\qquad(\mbox{by the Jensen operator inequality \eqref{ji}})\\
 & = \int_T\Phi_t((f\Delta h)(A_t,B_t))d\mu(t),
 \end{align*}
 which gives rise to \eqref{qqi1}.
\end{proof}
As special cases of Theorem \ref{th10} we have the following result.
\begin{corollary} \label{c10}
 Let $f$ be an operator convex function with $f(0)\leq0$ and $h$ be an operator concave function with $h>0$. If $\Phi$ is a positive linear map on $\mathbb{B}(\mathscr{H})$ with $\Phi(I)\leq I$, then
 \begin{eqnarray}\label{qq1}
 (f\Delta h)\left(\Phi(A),\Phi(B)\right)\leq\Phi((f\Delta h)(A,B)),
 \end{eqnarray}
 for all self-adjoint operators $A,B$. In particular, if $g$ is the perspective function associated to $f$, then
 \begin{eqnarray}\label{qq2}
 g\left(\Phi(A),\Phi(B)\right)\leq\Phi(g(A,B)),
 \end{eqnarray}
 for any self-adjoint operator $A$ and any strictly positive operator $B$.
\end{corollary}
\begin{example}
 Let $\Phi$ be a positive linear map on $\mathbb{B}(\mathscr{H})$. Applying Corollary \ref{c10} to the operator convex function $f(t)=t^\beta$ \,\, $(-1\leq \beta\leq0\ \mbox{or}\ 1\leq\beta\leq2)$ and the operator concave function $h(t)=t^\alpha$ \,\, $(0\leq\alpha\leq1)$, we obtain
 \begin{eqnarray}\label{q10}
 \Phi(B)^{\frac{\alpha}{2}}\left(\Phi(B)^{-\frac{\alpha}{2}}\Phi(A)
 \Phi(B)^{\frac{-\alpha}{2}}\right)^\beta\Phi(B)^{\frac{\alpha}{2}}\leq
 \Phi\left(B^{\frac{\alpha}{2}}\left(B^{-\frac{\alpha}{2}}AB^{-\frac{\alpha}{2}}\right)^\beta
 B^{\frac{\alpha}{2}}\right).
 \end{eqnarray}
In particular, for $\alpha=\frac{1}{2}$ and $\beta=-1$, inequality
\eqref{q10} gives rise to
\begin{eqnarray*}
 \Phi(B)^{\frac{1}{2}}\Phi(A)^{-1}\Phi(B)^{\frac{1}{2}}\leq\Phi\left(B^{\frac{1}{2}}A^{-1}
 B^{\frac{1}{2}}\right).
\end{eqnarray*}
Note that with $\alpha=1$ and $\beta=-1$, inequality \eqref{q10}
gives rise to the known inequality
\begin{eqnarray*}
\Phi(B)\Phi(A)^{-1}\Phi(B)\leq\Phi\left(BA^{-1}
 B\right).
\end{eqnarray*}
\end{example}
\begin{corollary}
 Let $f$ be an operator convex function with $f(0)\leq0$ and $h$ be an operator concave function with $h>0$. Then
 \begin{eqnarray*}
 (f\Delta h)\left(\langle Ax,x\rangle,\langle Bx,x\rangle\right)\leq\langle(f\Delta h)(A,B)x,x\rangle,
 \end{eqnarray*}
 for all self-adjoint operators $A, B\in\mathbb{\mathbb{B}}(\mathscr{H})$ and all unit vectors $x\in \mathscr{H}$. In particular, if $g$ is the perspective function of $f$, then
 \begin{eqnarray}\label{qq11}
 g\left(\langle Ax,x\rangle,\langle Bx,x\rangle\right)\leq\langle g(A,B)x,x\rangle,
 \end{eqnarray}
 for any self-adjoint operator $A \in\mathbb{\mathbb{B}}(\mathscr{H})$, any strictly positive operator $B \in\mathbb{\mathbb{B}}(\mathscr{H})$ and any unit vector $x\in \mathscr{H}$.
\end{corollary}

\begin{theorem}
Let $f_1$ and $f_2$ be operator convex functions with $f_1(0)\leq0$
and $f_2(0)\leq0$ and let $h$ be an operator concave function with
$h>0$. The following assertions are equivalent:
\begin{enumerate}
 \item $f_1\leq f_2$;\\
 \item $(f_1\Delta h)\left(\int_T\Phi_t(A_t)d\mu(t),\int_T\Phi_t(B_t)d\mu(t))\right)\leq\int_T\Phi_t((f_2\Delta h)(A_t,B_t))d\mu(t)$ for all unital fields $(\Phi_t)_{t\in T}$ and all continuous fields of operators $(A_t)_{t\in T}$ and $(B_t)_{t\in T}$;\\
 \item $f_1\left(\int_T\Phi_t(A_t)d\mu(t)\right)\leq\int_T\Phi_t(f_2(A_t))d\mu(t)$ for all continuous fields of operators $(A_t)_{t\in T}$.
  \end{enumerate}
\end{theorem}
\begin{proof}
$(1)\Rightarrow (2)$ \ Let $f_1\leq f_2$ and $(B_t)_{t\in T}$ be a
continuous field of self-adjoint operators in $\mathfrak{A}$. Let
$C=\int_T\Phi_t(B_t)d\mu(t)$. Define the field of positive linear
maps $(\Psi_t)_{t\in T}:\mathfrak{A}\to \mathfrak{B}$ by
\begin{eqnarray*}
 \Psi_t(X)=h(C)^{-\frac{1}{2}}
 \Phi_t\left(h(B_t)^{\frac{1}{2}}X
h(B_t)^{\frac{1}{2}}\right)
 h(C)^{-\frac{1}{2}}.
\end{eqnarray*}
It follows from the operator concavity of $h$ and $h>0$ that
\begin{align*}
 \int_T\Phi_t(h(B_t))d\mu(t)
\leq h\left(\int_T\Phi_t(B_t)d\mu(t)\right).
\end{align*}
 Therefore
\begin{eqnarray*}
 \int_T\Psi_t(I)d\mu(t)=\int_Th(C)^{-\frac{1}{2}}\Phi_t(h(B_t))h(C)^{-\frac{1}{2}}\leq I.
\end{eqnarray*}
Hence
\begin{align*}
(f_1\Delta h)&\left(\int_T\Phi_t(A_t)d\mu(t),\int_T\Phi_t(B_t)d\mu(t))\right)\\
 &=h(C)^{\frac{1}{2}} f_1\left(h(C)^{-\frac{1}{2}}\int_T\Phi_t(A_t)d\mu(t)
 h(C)^{-\frac{1}{2}}\right)h(C)^{\frac{1}{2}}\\
 &=h(C)^{\frac{1}{2}}f_1\left(\int_T\Psi_t
 \left(h(B_t)^{-\frac{1}{2}}A_th(B_t)^{-\frac{1}{2}}\right)d\mu(t)
 \right) h(C)^{\frac{1}{2}}\\
& \leq h(C)^{\frac{1}{2}}f_2\left(\int_T\Psi_t
 \left(h(B_t)^{-\frac{1}{2}}A_th(B_t)^{-\frac{1}{2}}\right)d\mu(t)
 \right) h(C)^{\frac{1}{2}}\\
 & \leq h(C)^{\frac{1}{2}}\int_T\Psi_t\left(f_2\left(h(B_t)^{-\frac{1}{2}}A_th(B_t)^{-\frac{1}{2}}\right)
 \right) h(C)^{\frac{1}{2}}\\
 & \qquad\qquad\qquad\qquad\qquad(\mbox{by the Jensen operator inequality \eqref{ji}})\\
 &= \int_T\Phi_t((f_2\Delta h)(A_t,B_t)).
\end{align*}
$(2)\Rightarrow (3)$ Let $h(t)=t$. Then
\begin{eqnarray*}
 f_1\left(\int_T\Phi_t(A_t)d\mu(t)\right)&=&(f_1\Delta h)\left(\int_T\Phi_t(A_t)d\mu(t),\int_T\Phi_t(I)d\mu(t)\right)\\
 &\leq& \int_T\Phi_t((f_2\Delta h)(A_t,I))d\mu(t)=\int_T\Phi_t(f_2(A_t))d\mu(t).
\end{eqnarray*}
$(3)\Rightarrow (1)$ Putting $T=\{1\}$ and $\Phi_1(A)=A$ in $(3)$ we
get $(1)$.
\end{proof}
Applying the theorem above to $h(t)=t$ we get the following result.
\begin{corollary}\label{qqq2}
 Let $f_1$ and $f_2$ be operator convex functions and $g_1$ and $g_2$ be the corresponding operator perspective functions, respectively. The following assertions are equivalent:
 \begin{enumerate}
 \item $f_1\leq f_2$,\\
 \item $g_1\left(\Phi(A),\Phi(B)\right)\leq\Phi(g_2(A,B))$ for any unital positive linear map $\Phi$ on $\mathbb{B}(\mathscr{H})$, any self-adjoint operator $A \in \mathbb{B}(\mathscr{H})$ and any strictly positive operator $B \in \mathbb{B}(\mathscr{H})$.\\
 \item $f_1(\Phi(A))\leq\Phi(f_2(A))$ for any unital positive linear map $\Phi$ on $\mathbb{B}(\mathscr{H})$ and any self-adjoint operator $A\in\mathbb{B}(\mathscr{H})$.
  \end{enumerate}
\end{corollary}


In the next theorem, we establish a relation between two functions
$f\Delta h$ and $f \nabla h$.

\begin{theorem}\label{1111}
 Let $f$ be an operator convex function with $f(0)\leq0$ and $h$ be an operator concave function with $h>0$. If $\mu$ is a probability measure on $T$, then
 \begin{eqnarray}\label{qq30}
  (f\Delta h)(A,B)\leq(f\nabla h)(\widetilde{A},\widetilde{B}),
 \end{eqnarray}
 where $A=\int_TA_td\mu(t)$ and $B=\int_TB_td\mu(t)$.
\end{theorem}
\begin{proof}
 \begin{align}\label{q8}
 &f\left(h(B)^{-\frac{1}{2}}A h(B)^{-\frac{1}{2}}\right)\nonumber\\
 & = f\left(h\left(\int_TB_sd\mu(s)\right)^{-\frac{1}{2}}\left(\int_TA_td\mu(t)\right)
 h\left(\int_TB_sd\mu(s)\right)^{-\frac{1}{2}}\right)\nonumber\\
 &= f\left(\int_Th\left(\int_TB_sd\mu(s)\right)^{-\frac{1}{2}}A_t
 h\left(\int_TB_sd\mu(s)\right)^{-\frac{1}{2}}d\mu(t)\right)\nonumber\\
 &= f\big(\int_Th\left(\int_TB_sd\mu(s)\right)^{-\frac{1}{2}}h(B_t)^{\frac{1}{2}}
 \left(h(B_t)^{-\frac{1}{2}}A_th(B_t)^{-\frac{1}{2}}\right)h(B_t)^{\frac{1}{2}}h\left(\int_TB_sd\mu(s)\right)^{-\frac{1}{2}}d\mu(t)\big).
 \end{align}
 Since $h$ is operator concave,
 \begin{eqnarray*}
\int_Th(B_t)d\mu(t)\leq h\left(\int_TB_td\mu(t)\right).
 \end{eqnarray*}
 So, \eqref{q8}, the operator convexity of $f$ and $f(0)\leq0$ imply that
 \begin{align*}
 f&\left(h(B)^{-\frac{1}{2}}A h(B)^{-\frac{1}{2}}\right)\nonumber\\
 & =f\left(\int_Th(B)^{-\frac{1}{2}}h(B_t)^{\frac{1}{2}}
 \left(h(B_t)^{-\frac{1}{2}}A_th(B_t)^{-\frac{1}{2}}\right)h(B_t)^{\frac{1}{2}}
 h(B)^{-\frac{1}{2}}d\mu(t)\right)\\
 & \leq \int_Th(B)^{-\frac{1}{2}}h(B_t)^{\frac{1}{2}}
 f\left(h(B_t)^{-\frac{1}{2}}A_th(B_t)^{-\frac{1}{2}}\right)
 h(B_t)^{\frac{1}{2}}
 h(B)^{-\frac{1}{2}}d\mu(t)\\
 &= h(B)^{-\frac{1}{2}}\int_Th(B_t)^{\frac{1}{2}}
 f\left(h(B_t)^{-\frac{1}{2}}A_th(B_t)^{-\frac{1}{2}}\right)
 h(B_t)^{\frac{1}{2}}d\mu(t)\, h(B)^{-\frac{1}{2}}\\
 &= h(B)^{-\frac{1}{2}}(f\nabla h)(\widetilde{A},\widetilde{B})h(B)^{-\frac{1}{2}},
 \end{align*}
 whence we get the required inequality \eqref{qq30}.
\end{proof}
 In the discrete case, assume that $\widetilde{p}=(p_1,\cdots,p_n)$ and $\widetilde{q}=(q_1,\cdots,q_n)$ are probability distributions. In this case let us define $f\nabla h$ by
\begin{eqnarray*}
(f\nabla
h)(\widetilde{L},\widetilde{R},\widetilde{p},\widetilde{q})=\sum_{i=1}^{n}p_i
h(q_iR_i)^{\frac{1}{2}} f\left(h(q_iR_i)^{-\frac{1}{2}}L_i
h(q_iR_i)^{-\frac{1}{2}}\right)h(q_iR_i)^{\frac{1}{2}}.
\end{eqnarray*}
Note that with $p_1=q_1=1$ and $p_i=q_i=0$\,\ $(i=2,\cdots,n)$,
$f\nabla h=f\Delta h$.

\begin{remark}We can generalize \eqref{qq1} to $f\nabla h$ as follows. If $f$ is operator convex with $f(0)\leq0$, $h$ is operator concave with $h>0$, and $\Phi$ is a unital positive linear map on $\mathbb{B}(\mathscr{H})$, then
 \begin{eqnarray*}
(f\nabla
h)(\widetilde{\Phi(L)},\widetilde{\Phi(R)},\widetilde{p},\widetilde{q})\leq\Phi((f\nabla
h)(\widetilde{L},\widetilde{R},\widetilde{p},\widetilde{q})),
\end{eqnarray*}
 where $\widetilde{\Phi(L)}=(\Phi(L_1),\cdots,\Phi(L_n))$ and $\widetilde{\Phi(R)}=(\Phi(R_1),\cdots,\Phi(R_n))$.

\end{remark}
The following result similar to \eqref{qq30} may be stated.
\begin{corollary}\label{1111}
 Let $f$ be an operator convex function with $f(0)<0$ and $h$ be an operator concave function with $h>0$. If $\widetilde{p}=(p_1,\cdots,p_n)$ and $\widetilde{q}=(q_1,\cdots,q_n)$ are probability distributions, then
 \begin{eqnarray}\label{q30}
  (f\Delta h)(L,R)\leq(f\nabla h)(\widetilde{L},\widetilde{R},\widetilde{p},\widetilde{q}),
 \end{eqnarray}
 where $L=\sum_{i=1}^{n}p_iL_i$, $R=\sum_{i=1}^{n}q_iR_i$.
\end{corollary}
The next theorem, gives a bound for the non-commutative
$f$-divergence functional, when $f$ is differentiable.
\begin{theorem}\label{th4}
 Let $g$ and $\Theta$ be the perspective function and the non-commutative $f$-divergence functional associated to a differentiable operator convex function $f$, respectively. Then
 \begin{eqnarray}\label{q4}
 f(I)\int_TB_td\mu(t)-f'(I)\int_T(B_t-A_t)d\mu(t)\leq \Theta(\widetilde{A},\widetilde{B}).
 \end{eqnarray}
\end{theorem}
\begin{proof}
 It follows from the convexity of $f$ that
 \begin{eqnarray}\label{q3}
 f(y)-f(x)\leq f'(y)(y-x),
 \end{eqnarray}
 for each $x,y$. Using the functional calculus to \eqref{q3} we obtain
 \begin{eqnarray}\label{q115}
 f(I)-f\left(B_t^{-\frac{1}{2}}A_tB_t^{-\frac{1}{2}}\right)\leq f'(I)\left(I-B_t^{-\frac{1}{2}}A_tB_t^{-\frac{1}{2}}\right),
 \end{eqnarray}
 for each $t\in T$.
 Multiplying both sides of \eqref{q115} by $B_t^{\frac{1}{2}}$ and integrating over $t$ we get
 \begin{eqnarray*}
 f(I)\int_TB_td\mu(t)- \Theta(\widetilde{A},\widetilde{B})\leq f'(I)\int_T(B_t-A_t)d\mu(t),
 \end{eqnarray*}
 which ensures \eqref{q4}.
\end{proof}
\begin{corollary}\label{d}
 If $f$ is operator convex and differentiable and $g$ is the perspective function of $f$, then
 \begin{eqnarray*}
 f(I)\sum_{i=1}^{n}R_i-f'(I)\sum_{i=1}^{n}(R_i-L_i)\leq \sum_{i=1}^{n}g(L_i,R_i).
 \end{eqnarray*}
\end{corollary}
\section{Applications}
In this section, we use the results of section 2 to derive some
interesting operator inequalities.
\subsection{Refinements of Choi--Davis--Jensen inequality}

 Let $T_1$ and $T_2$ be disjoint locally compact Hausdorff spaces, $T_1\cup T_2=T$ and $\mu$ be a bounded Radon measure on $T$.
 As the first application of our result in section 2, we obtain the following refinement of the Jensen operator inequality \eqref{ji}.
 \begin{theorem}\label{th1}
 Let $f$ be an operator convex function, $(A_t)_{t\in T}$ be a continuous field of self-adjoint operators in $\mathfrak{A}$, $(\Phi_t)_{t\in T}:\mathfrak{A}\to\mathfrak{B}$ be a unital field of positive linear maps, $D_{T_1}=\int_{T_1}\Phi_t(I)d\mu(t)$ and
 $D_{T_2}=\int_{T_2}\Phi_t(I)d\mu(t)$. Then
 \begin{align}\label{qj1}
  {\rm (i)} \ f&\left(\int_T\Phi_t(A_t)d\mu(t)\right)\nonumber\\
 &\leq D_{T_1}^{\frac{1}{2}} f\left(D_{T_1}^{-\frac{1}{2}}\int_{T_1}\Phi_t(A_t)d\mu(t)\,D_{T_1}
 ^{-\frac{1}{2}}\right) D_{T_1}^{\frac{1}{2}}
  + D_{T_2}^{\frac{1}{2}} f\left(D_{T_2}^{-\frac{1}{2}}
 \int_{T_2}\Phi_t(A_t)d\mu(t)\,D_{T_2}^{-\frac{1}{2}}\right)
 D_{T_2}^{\frac{1}{2}}\nonumber\\
 & \leq\int_T\Phi_t(I)^{\frac{1}{2}}f\left(\Phi_t(I)^{-\frac{1}{2}}
  \Phi_t(A_t)\Phi_t(I)^{-\frac{1}{2}}\right)\Phi_t(I)^{\frac{1}{2}}d\mu(t)\nonumber\\
 &\leq \int_T\Phi_t(f(A_t))d\mu(t).
 \end{align}
 \begin{eqnarray}
  {\rm (ii)} \ \int_T\Phi_t(f(A_t))d\mu(t) & - & f\left(\int_T\Phi_t(A_t)d\mu(t)\right)\nonumber\\
  & \geq&
 \int_{T_1}\Phi_t(f(A_t))d\mu(t) - D_{T_1}^{\frac{1}{2}}f
 \left(D_{T_1}^{-\frac{1}{2}}\int_{T_1} \Phi_t(A_t)D_{T_1}^{-\frac{1}{2}}d\mu(t)\right)D_{T_1}^{\frac{1}{2}}\nonumber\\
 & \geq& 0.
 \end{eqnarray}
\end{theorem}
\begin{proof}
(i) \ Put $C_1=D_{T_1}^{\frac{1}{2}}$ and
$C_2=D_{T_2}^{\frac{1}{2}}$. Clearly $C_1^*C_1+C_2^*C_2=I$. It
follows from \eqref{multij} that
 \begin{align*}
 & D_{T_1}^{\frac{1}{2}}f\left(D_{T_1}^{-\frac{1}{2}}\int_{T_1}
 \Phi_t(A_t)d\mu(t)\,D_{T_1}^{-\frac{1}{2}}\right)D_{T_1}^{\frac{1}{2}}
 + D_{T_2}^{\frac{1}{2}} f\left(D_{T_2}^{-\frac{1}{2}}\int_{T_2}\Phi_t(A_t)d\mu(t)\,D_{T_2}^{-\frac{1}{2}}\right)
 D_{T_2}^{\frac{1}{2}}\\
 &= C_1^*f\left({C_1^*}^{-1}\int_{T_1}\Phi_t(A_t)d\mu(t)\,C_1^{-1}\right)C_1+ C_2^*f\left({C_2^*}^{-1}\int_{T_2}\Phi_t(A_t)d\mu(t)\,C_2^{-1}\right)C_2\\
  & \geq f\left(\int_{T_1}\Phi_t(A_t)d\mu(t)+\int_{T_2}\Phi_t(A_t)d\mu(t)\right)d\mu(t)\\
 & = f\left(\int_T\Phi_t(A_t)d\mu(t)\right),
  \end{align*}
  which is the first inequality of \eqref{qj1}.
 Assume that $g$ is the perspective function of $f$. It follows from Theorem \ref{th2} that
 \begin{align*}
 D_{T_1}^{\frac{1}{2}}& f\left(D_{T_1}^{-\frac{1}{2}}\int_{T_1}\Phi_t(A_t)d\mu(t)\,D_{T_1}^{-\frac{1}{2}}\right)
 D_{T_1}^{\frac{1}{2}}+ D_{T_2}^{\frac{1}{2}} f\left(D_{T_2}^{-\frac{1}{2}}\int_{T_2}\Phi_t(A_t)d\mu(t)\,D_{T_2}^{-\frac{1}{2}}\right)D_{T_2}^{\frac{1}{2}}\\
  &=g\left(\int_{T_1}\Phi_t(A_t)d\mu(t),D_{T_1}\right)+
  g\left(\int_{T_2}\Phi_t(A_t)d\mu(t),D_{T_2}\right)\\
  &= g\left(\int_{T_1}\Phi_t(A_t)d\mu(t),\int_{T_1}\Phi_t(I)d\mu(t)\right)
  +g\left(\int_{T_2}\Phi_t(A_t)d\mu(t),\int_{T_2}\Phi_t(I)d\mu(t)\right)\\
  & \leq \int_{T_1}g(\Phi_t(A_t),\Phi_t(I))d\mu(t)+\int_{T_2}g(\Phi_t(A_t),\Phi_t(I))d\mu(t)
  \qquad\qquad\quad (\mbox{by \eqref{q2}})\\
  &= \int_Tg(\Phi_t(A_t),\Phi_t(I))d\mu(t)\\
  & = \int_T\Phi_t(I)^{\frac{1}{2}}f\left(\Phi_t(I)^{-\frac{1}{2}}\Phi_t(A_t)
  \Phi_t(I)^{-\frac{1}{2}}\right)
  \Phi_t(I)^{\frac{1}{2}}d\mu(t),
 \end{align*}
 whence we get the second inequality of \eqref{qj1}. For each $t\in T$, let the unital positive linear map $\Psi_t:\mathfrak{A}\to\mathfrak{B}$ be defined by
 \begin{eqnarray*}
 \Psi_t(X)=\Phi_t(I)^{-\frac{1}{2}}\Phi_t(X)\Phi_t(I)^{-\frac{1}{2}}.
 \end{eqnarray*}
 Since $f$ is operator convex, we have
 \begin{eqnarray}\label{qj2}
 f\left(\Phi_t(I)^{-\frac{1}{2}}\Phi_t(A_t)
  \Phi_t(I)^{-\frac{1}{2}}\right)&=&f(\Psi_t(A_t))\nonumber\\
  &\leq&\Psi_t(f(A_t))\nonumber\\
  &=&\Phi_t(I)^{-\frac{1}{2}}\Phi_t(f(A_t))\Phi_t(I)^{-\frac{1}{2}}.
 \end{eqnarray}
 The last inequality of \eqref{qj1} now follows from \eqref{qj2}. \\
 (ii) \ Let $\Psi$ be the unital positive linear map defined by
 \begin{align*}
 \Psi\left((A_t)_{t\in T}\oplus B\right)=\int_{T_2} \Phi_t(A_t)d\mu(t)+D_{T_1}^{\frac{1}{2}}BD_{T_1}^{\frac{1}{2}}.
 \end{align*}
 Applying Choi--Davis--Jensen's inequality \eqref{choi} to $\Psi$ we obtain
\begin{align*}
 f\left(\int_T\Phi_t(A_t)d\mu(t)\right)&= f\left(\int_{T_2}\Phi_t(A_t)d\mu(t)+ D_{T_1}^{\frac{1}{2}}\left(D_{T_1}^{-\frac{1}{2}}\int_{T_1}\Phi_t(A_t)d\mu(t)
 D_{T_1}^{-\frac{1}{2}}\right)D_{T_1}^{\frac{1}{2}} \right)\\
 & = f\left(\Psi\left(\int_{T_2}A_td\mu(t)\oplus D_{T_1}^{-\frac{1}{2}}\int_{T_1}\Phi_t(A_t)d\mu(t)
 D_{T_1}^{-\frac{1}{2}}\right) \right)\\
 &\leq \int_{T_2}\Phi_t(f(A_t))d\mu(t)+ D_{T_1}^{\frac{1}{2}}f\left(D_{T_1}^{-\frac{1}{2}}\int_{T_1}\Phi_t(A_t)d\mu(t)
 D_{T_1}^{-\frac{1}{2}}\right)D_{T_1}^{\frac{1}{2}} .
\end{align*}
Hence
\begin{align*}
 \int_T\Phi_t(f(A_t))d\mu(t) &- f\left(\int_T\Phi_t(A_t)d\mu(t)\right)\\
 & \geq \int_T\Phi_t(f(A_t))d\mu(t)-\int_{T_2}\Phi_t(f(A_t))d\mu(t)\\
 & \quad - D_{T_1}^{\frac{1}{2}}f\left(D_{T_1}^{-\frac{1}{2}}\int_{T_1}\Phi_t(A_t)d\mu(t)
 D_{T_1}^{-\frac{1}{2}}\right)D_{T_1}^{\frac{1}{2}}\\
 & = \int_{T_1}\Phi_t(f(A_t))d\mu(t)- D_{T_1}^{\frac{1}{2}}f\left(D_{T_1}^{-\frac{1}{2}}\int_{T_1}\Phi_t(A_t)d\mu(t)
 D_{T_1}^{-\frac{1}{2}}\right)D_{T_1}^{\frac{1}{2}}\\
 & \geq 0.
\end{align*}
The last inequality follows from \eqref{ji}.
\end{proof}
Assume that $\Phi_1,\cdots,\Phi_n$ be positive linear maps on
$\mathbb{B}(\mathscr{H})$ with $\sum_{i=1}^{n}\Phi_i(I)=I$ and
$A_1,\cdots,A_n$ be self-adjoint operators on $\mathscr{H}$.
Applying Theorem \ref{th1} to $T=\{1,\cdots,n\}$, $T_1\subseteq T$,
$T_2=T-T_1$ and the counting measure $\mu$, we have the following
consequence, which is a refinement of \eqref{jo}.
\begin{corollary}\label{co1}
 Let $f$ be an operator convex function, $D_{T_1}=\sum_{i\in {T_1}}\Phi_i(I)$ and
 $D_{T_2}=\sum_{i\in {T_2}}\Phi_i(I)$. Then
 \begin{align}\label{qqj1}
  {\rm (i)} \ f&\left(\sum_{i=1}^{n}\Phi_i(A_i)\right)\nonumber\\
 &\leq D_{T_1}^{\frac{1}{2}} f\left(D_{T_1}^{-\frac{1}{2}}\sum_{i\in {T_1}}\Phi_i(A_i)D_{T_1}^{-\frac{1}{2}}\right)D_{T_1}^{\frac{1}{2}}
 + D_{T_2}^{\frac{1}{2}} f\left(D_{T_2}^{-\frac{1}{2}}\sum_{i\in {T_2}}\Phi_i(A_i)D_{T_2}^{-\frac{1}{2}}\right)D_{T_2}^{\frac{1}{2}}\nonumber\\
 & \leq\sum_{i=1}^{n}\Phi_i(I)^{\frac{1}{2}}f\left(\Phi_i(I)^{-\frac{1}{2}}
  \Phi_i(A_i)\Phi_i(I)^{-\frac{1}{2}}\right)\Phi_i(I)^{\frac{1}{2}}\nonumber\\
 &\leq \sum_{i=1}^{n}\Phi_i(f(A_i));
 \end{align}
 \begin{align}
 {\rm (ii)} \ \sum_{i=1}^{n}\Phi_i(f(A_i)) &- f\left(\sum_{i=1}^{n}\Phi_i(A_i)\right) \nonumber\\
  &\geq \sum_{i\in {T_1}}\Phi_i(f(A_i))-D_{T_1}^{\frac{1}{2}}f\left(D_{T_1}^{-\frac{1}{2}}\sum_{i\in {T_1}}\Phi_i(A_i)D_{T_1}^{-\frac{1}{2}}\right)D_{T_1}^{\frac{1}{2}}\nonumber\\
 & \geq 0.
 \end{align}
\end{corollary}
We give an example to show that all inequalities of \eqref{qqj1} are
sharp. So either \eqref{qj1} or \eqref{qqj1} is really a refinement
of the Choi--Davis--Jensen inequality.
\begin{example}
 Let $f(t)=t^2$, $T=\{1,2,3\}$ and $T_1=\{1\}$. Consider the positive linear maps $\Phi_1,\Phi_2,\Phi_3:\mathcal{M}_3(\mathbb{C})\to\mathcal{M}_2(\mathbb{C})$ defined by
 \begin{eqnarray*}
 \Phi_1(A)=\frac{1}{3}(a_{ij})_{1\leq i,j\leq2},\qquad \Phi_2(A)=\Phi_3(A)=\frac{1}{3}(a_{ij})_{2\leq i,j\leq3},
 \end{eqnarray*}
 where $A\in\mathcal{M}_3(\mathbb{C})$. Clearly $\Phi_1(I_3)+\Phi_2(I_3)+\Phi_3(I_3)=I_2$, where $I_3$ and $I_2$ are the identity operators in
$\mathcal{M}_3(\mathbb{C})$ and $\mathcal{M}_2(\mathbb{C})$,
respectively.
 Also
 $D_{T_1}=\Phi_1(I_3)=\frac{1}{3}I_2$ and $D_{T_2}=\Phi_2(I_3)+\Phi_3(I_3)=\frac{2}{3}I_2$. If
\begin{eqnarray*}
 A_1=3\left(
\begin{array}{ccc}
2&0&1 \\
0&1&0 \\
1&0&0
\end{array}\right), \quad
 A_2=3\left(
\begin{array}{ccc}
0&0&1 \\
0&1&0 \\
1&0&0
\end{array}\right),\quad
A_3=3\left(
\begin{array}{ccc}
1&0&1 \\
0&0&1 \\
1&1&1
\end{array}\right),
\end{eqnarray*}
then
\begin{eqnarray*}
 \left(\Phi_1(A_1)+\Phi_2(A_2)+\Phi_3(A_3)\right)^2=\left(
\begin{array}{cc}
10&5 \\
5&5
\end{array}\right),
\end{eqnarray*}
\begin{align*}
 D_{T_1}^{\frac{1}{2}} f\left(D_{T_1}^{-\frac{1}{2}}\sum_{i\in{T_1}}\Phi_i(A_i)D_{T_1}^{-\frac{1}{2}}\right)
 D_{T_1}^{\frac{1}{2}}
  + D_{T_2}^{\frac{1}{2}} f\left(D_{T_2}^{-\frac{1}{2}}\sum_{i\in {T_2}}\Phi_i(A_i)D_{T_2}^{-\frac{1}{2}}\right)D_{T_2}^{\frac{1}{2}}
   = \left(
\begin{array}{cc}
15&3 \\
3&6
\end{array}\right),
\end{align*}
\begin{align*}
 \Phi_1(I)^{\frac{1}{2}}&\left(\Phi_1(I)^{-\frac{1}{2}}\Phi_1(A_1)
 \Phi_1(I)^{-\frac{1}{2}}\right)^2\Phi_1(I)^{\frac{1}{2}}\\
 & + \Phi_2(I)^{\frac{1}{2}}
 \left(\Phi_2(I)^{-\frac{1}{2}}\Phi_2(A_2)
 \Phi_2(I)^{-\frac{1}{2}}\right)^2\Phi_2(I)^{\frac{1}{2}}\\
 & + \Phi_3(I)^{\frac{1}{2}}\left(\Phi_3(I)^{-\frac{1}{2}}\Phi_3(A_3)
 \Phi_3(I)^{-\frac{1}{2}}\right)^2\Phi_3(I)^{\frac{1}{2}}\\
 &= \left(
\begin{array}{cc}
18&3 \\
3&9
\end{array}\right),
\end{align*}
\begin{eqnarray*}
 \Phi_1(f(A_1))+\Phi_2(f(A_2))+\Phi_3(f(A_3))=\left(
\begin{array}{cc}
21&3 \\
3&15
\end{array}\right).
\end{eqnarray*}
Now inequalities
\begin{eqnarray*}
 \left(
\begin{array}{cc}
10&5 \\
5&5
\end{array}\right)\lvertneqq \left(
\begin{array}{cc}
15&3 \\
3&6
\end{array}\right)\lvertneqq \left(
\begin{array}{cc}
18&3 \\
3&9
\end{array}\right)\lvertneqq \left(
\begin{array}{cc}
21&3 \\
3&15
\end{array}\right),
\end{eqnarray*}
show that all inequalities of \eqref{qj1} can be strict. By similar
computations, one may show that inequalities of (ii) are strict.
\end{example}
\begin{corollary}\label{cor1}
Let $f$ be an operator convex function, $A_1,\cdots,A_n$ be
self-adjoint operators and $C_1,\cdots,C_n$ be such that
$\sum_{i=1}^{n}C_i^*C_i=I$. Then
\begin{align*}
 f&\left(\sum_{i=1}^{n}C_i^*A_iC_i\right)\\
 &\leq C_{T_1}^{\frac{1}{2}}f\left(C_{T_1}^{-\frac{1}{2}}\sum_{i\in {T_1}} C_i^*A_iC_iC_{T_1}^{-\frac{1}{2}}\right)C_{T_1}^{\frac{1}{2}}
 + C_{T_2}^{\frac{1}{2}} f\left(C_{T_2}^{-\frac{1}{2}}\sum_{i\in{T_2}}C_i^*A_iC_iC_{T_2}^{-\frac{1}{2}}\right)
 C_{T_1}^{\frac{1}{2}}\\
 &\leq \sum_{i=1}^{n}(C_i^*C_i)^{\frac{1}{2}}f\left((C_i^*C_i)^{-\frac{1}{2}}
 (C_i^*A_iC_i)(C_i^*C_i)^{-\frac{1}{2}}\right)(C_i^*C_i)^{\frac{1}{2}}\\
 & \leq \sum_{i=1}^{n}C_i^*f(A_i)C_i,
 \end{align*}
 where $C_{T_1}=\sum_{i\in{T_1}}C_i^*C_i$ and $C_{T_2}=\sum_{i\in{T_2}}C_i^*C_i$.
\end{corollary}
\begin{proof}
 Apply Corollary \ref{co1} to $\Phi_i(A)=C_i^*AC_i$\,\, $(1\leq i\leq n)$.
\end{proof}


\subsection{Unitarily invariant norm inequalities}
Now we use the results of section 2 to obtain some norm
inequalities. For this end, we need to recall some preliminary
concepts. A norm $|||\cdot|||$ on $\mathcal{M}_n(\mathbb{C})$ is
said to be unitarily invariant if $|||UAV|||=|||A|||$, for any
$A\in\mathcal{M}_n(\mathbb{C})$ and all unitaries
$U,V\in\mathcal{M}_n(\mathbb{C})$. For a Hermitian matrix
$A\in\mathcal{M}_n(\mathbb{C})$, we denote by $\lambda_1(A) \geq
\lambda_2(A) \geq \cdots\lambda_n(A)$ the eigenvalues of $A$
arranged in the decreasing order with their multiplicities counted.
By $s_1(A)\geq s_2(A) \geq \cdots\geq s_n(A)$ we denote the
eigenvalues of $|A|=(A^*A)^{1/2}$, i.e., the singular values of $A$.
One of important classes of unitarily invariant norms is the class
of the Ky Fan $k$-norms defined by
 \begin{eqnarray*}
 |||A|||_{(k)}=\sum_{i=1}^{k}s_i(A),\qquad 1\leq k\leq n.
 \end{eqnarray*}
We need the following lemmas.
\begin{lemma}\cite[Theorem III.3.5]{Bh}\label{minimax}
 Let $A\in\mathcal{M}_n(\mathbb{C})$. Then
 \begin{eqnarray*}
 \sum_{i=1}^{k}\lambda_i(A)=\max \sum_{i=1}^{k}\langle Au_i,u_i\rangle \quad\qquad
 (k=1,\cdots,n),
 \end{eqnarray*}
 where the maximum is taken over all choices of orthonormal vectors $u_1,\cdots,u_k\in\mathbb{C}^n$ under the usual inner product $\langle\cdot,\cdot\rangle$.
\end{lemma}
\begin{lemma}\cite[Theorem IV.2.2]{Bh}\label{fan}
 Let $A$ and $B$ be two matrices. Then $|||A|||_{(k)}\leq |||B|||_{(k)}$ for all $k=1,\cdots,n$ if and only if $ |||A|||\leq |||B|||$ for all unitarily invariant norms.
\end{lemma}
The following lemma is an extension of the Jenesen inequality to
separately convex functions of two variables.
\begin{lemma}\cite[Lemma 2.2]{MA}\label{jadjit}
 Let $f:[a,b]\times[c,d]\to\mathbb{R}$ be a separately convex function and $A,B\in\mathcal{M}_n(\mathbb{C})$. Then
 \begin{eqnarray*}
 f(\langle Au,u\rangle,\langle Bv,v\rangle)\leq\langle f(A,B)u\otimes v,u\otimes v\rangle,
 \end{eqnarray*}
 for all unit vectors $u,v\in\mathbb{C}^n$.
\end{lemma}
\begin{theorem}\label{norm2}
 Let $f$ be an operator convex function and $g$ be the perspective function of $f$. Then \begin{eqnarray}\label{norm1}
g(|||A|||,|||B|||)\leq||| g(A,B)|||,
 \end{eqnarray}
 for all unitarily invariant norms $|||\cdot|||$ and positive-definite matrices $A,B\in \mathcal{M}_n(\mathbb{C})$.
\end{theorem}
 \begin{proof} Let $\lambda_i(A), \mu_i(B)$ denote the eigenvalues of $A, B$, respectively. We have
 \begin{eqnarray*}
 g\left(\sum_{i=1}^{n}\lambda_i(A),\sum_{i=1}^{n}\mu_i(B)\right)&=& g\left(\sum_{i=1}^{n}\langle Au_i,u_i\rangle,\sum_{i=1}^{n}\langle
 Bv_i,v_i\rangle\right)\\
  &\leq&\sum_{i=1}^{n}g\left(\langle Au_i,u_i\rangle,\langle Bv_i,v_i\rangle\right) \qquad\qquad (\mbox{by \eqref{jd1}}) \\
  &\leq& \sum_{i=1}^{n}\langle g(A,B)u_i\otimes v_i,u_i\otimes v_i\rangle\quad(\mbox{by Lemma \ref{jadjit}})\\
  &\leq&\sum_{i=1}^{n}\nu_i(g(A,B))\qquad\qquad\qquad \ (\mbox{by Lemma \ref{minimax}}).
 \end{eqnarray*}
 Now, \eqref{norm1} follows from Lemma \ref{fan}.
 \end{proof}

 \begin{example}
Applying Theorem \ref{norm2} to the operator convex function
$f(t)=t^\beta$ \,\, $(-1\leq \beta\leq0\ \mbox{or}\
1\leq\beta\leq2)$, we obtain
\begin{eqnarray*}
 |||B|||^{\frac{1}{2}}\left(|||B|||^{-\frac{1}{2}}|||A|||\,|||B|||^{-\frac{1}{2}}\right)^\beta|||B|||^{\frac{1}{2}}\leq
 |||B^{\frac{1}{2}}\left(B^{-\frac{1}{2}}AB^{-\frac{1}{2}}\right)^\beta B^{\frac{1}{2}} |||,
\end{eqnarray*}
for all strictly positive matrices
$A,B\in\mathcal{M}_n(\mathbb{C})$. In particular,
\begin{eqnarray*}
|||A|||\,|||B|||^{-1}|||A|||\leq|||AB^{-1}A|||,
\end{eqnarray*}
for all unitarily invariant norms $|||\cdot|||$.
 \end{example}
\subsection{Kullback--Leibler distance}

 The Kullback--Leibler distance is obtained from $f$-divergence functional \eqref{def22}, where $f(t)=-\log t$ and is defined by
\begin{eqnarray*}
 KL(\widetilde{p},\widetilde{q}):=\sum_{i=1}^{n}p_i\log\left(\frac{p_i}{q_i}\right)\,,
\end{eqnarray*}
where $\widetilde{p}=(p_1,\cdots,p_n)$ and
$\widetilde{q}=(q_1,\cdots,q_n)$ are probability distributions. By
definition \eqref{def222}, the non-commutative $f$-divergence
functional $\Theta$, which we will call it "the operator
Kullback--Leibler functional", is defined by
\begin{eqnarray*}
 \Theta(\widetilde{L},\widetilde{R}):=
 \sum_{i=1}^{n}R_i^{\frac{1}{2}}\log\left(R_i^{\frac{1}{2}}L_i^{-1}R_i^{\frac{1}{2}}\right)R_i^{\frac{1}{2}}.
\end{eqnarray*}
Applying Corollary \ref{qw1} to $f(t)=-\log t$, we get
\begin{align*}
 \sum_{i=1}^{n}R_i^{\frac{1}{2}}&\log\left(R_i^{\frac{1}{2}}L_i^{-1}R_i^{\frac{1}{2}}\right)
 R_i^{\frac{1}{2}}\\
 & \leq \left(\sum_{i=1}^{n}R_i\right)^{\frac{1}{2}}
 \log\left(\left(\sum_{i=1}^{n}R_i\right)^{\frac{1}{2}}\left(\sum_{i=1}^{n}L_i\right)^{-1}\left(\sum_{i=1}^{n}R_i\right)^{\frac{1}{2}}\right)
 \left(\sum_{i=1}^{n}R_i\right)^{\frac{1}{2}}.
\end{align*}
In particular, for strictly positive operators $A,B,C,D$, we have
\begin{align*}
 A^{\frac{1}{2}}&\log\left(A^{\frac{1}{2}}C^{-1}A^{\frac{1}{2}}\right)A^{\frac{1}{2}}+
 B^{\frac{1}{2}}\log\left(B^{\frac{1}{2}}D^{-1}B^{\frac{1}{2}}\right)B^{\frac{1}{2}}\\
 & \leq (A+B)^{\frac{1}{2}}\log\left((A+B)^{\frac{1}{2}}(C+D)^{-1}(A+B)^{\frac{1}{2}}\right)(A+B)^{\frac{1}{2}}.
\end{align*}
Moreover, $f'(t)=-1/t$. Using Corollary \eqref{d} we get
\begin{eqnarray*}
 \sum_{i=1}^{n}(R_i-L_i)\leq \Theta(\widetilde{L},\widetilde{R}),
\end{eqnarray*}
or equivalently
\begin{eqnarray*}
 \sum_{i=1}^{n}R_i\leq \sum_{i=1}^{n}R_i^{\frac{1}{2}}\log\left(R_i^{\frac{1}{2}}L_i^{-1}R_i^{\frac{1}{2}}\right)
 R_i^{\frac{1}{2}}+\sum_{i=1}^{n}L_i.
\end{eqnarray*}
In particular, for $\widetilde{L}=(C,D)$ and $\widetilde{R}=(A,B)$,
we obtain
\begin{eqnarray*}
 A+B\leq C+D+ A^{\frac{1}{2}}\log\left(A^{\frac{1}{2}}C^{-1}A^{\frac{1}{2}}\right)A^{\frac{1}{2}}+
 B^{\frac{1}{2}}\log\left(B^{\frac{1}{2}}D^{-1}B^{\frac{1}{2}}\right)B^{\frac{1}{2}}.
\end{eqnarray*}
The function $f(t)=t\log t$ is operator convex and $f'(t)=\log t
+1$. Again, it follows from Corollary \ref{d} that
\begin{eqnarray*}
 \sum_{i=1}^{n}(L_i-R_i)\leq\sum_{i=1}^{n}L_iR_i^{-\frac{1}{2}}
 \log\left(R_i^{-\frac{1}{2}}L_iR_i^{-\frac{1}{2}}\right)R_i^{\frac{1}{2}}.
\end{eqnarray*}
In particular
\begin{eqnarray*}
 C+D\leq CA^{-\frac{1}{2}}\log\left(A^{-\frac{1}{2}}CA^{-\frac{1}{2}}\right)A^{\frac{1}{2}}+
 DB^{-\frac{1}{2}}\log\left(B^{-\frac{1}{2}}DB^{-\frac{1}{2}}\right)B^{\frac{1}{2}}+A+B.
\end{eqnarray*}


\begin{thebibliography}{99}
\bibitem{A} G.A. Anastassiou, Higher order optimal approximation of Csiszar's $f$-divergence, Nonlinear Anal., TMA, {\bf 61},  309-–339(2005).

\bibitem {Bh} R. Bhatia, Matrix Analysis, Springer-Verlag, New York (1997).

\bibitem{CD} P. Cerone and S.S. Dragomir, Approximation of the integral mean divergence and $f$-divergence via mean results, Math. Comput. Modelling, {\bf 42}, 207-219(2005).

\bibitem{C} I. Csisz\'{a}r, Information measures: A critical survey, Trans. 7th Prague Conf. on Info. Th., Statist. Decis. Funct., Random Processes and 8th European Meeting of Statist., Volume B, Academia Prague, 73-86 (1978).

\bibitem{C2} I. Csisz\'{a}r, Information-type measures of difference of probability distributions and indirect observations, Studia Sci. Math. Hungar, {\bf 2}, 299-318(1967).

\bibitem {CK} I. Csisz\'{a}r and J. K\"{o}rner, Information Theory: Coding Theorems for Discrete Memory-less Systems, Academic Press, New York (1981).

\bibitem{DK} S.S. Dragomir and S. Koumandos, Some inequalities for $f$-divergence measures generated by $2n$-convex functions, Acta Sci. Math. (Szeged), {\bf 76},  71-86(2010).

\bibitem{Eb} A. Ebadian, E. Nikoufar and M.E. Gordji, Perspectives of matrix convex functions, Proc. Natl. Acad. Sci. USA, {\bf 108}, 7313-7314(2011).

\bibitem{E} E.G. Effros,  A matrix convexity approach to some celebrated quantum inequalities, Proc. Natl. Acad. Sci. USA, {\bf 106}, 1006-1008(2009).

\bibitem{Fu} T. Furuta, H. Mi\'ci\'c, J. Pe\v{c}ari\'{c} and Y. Seo, Mond-Pecaric Method in Operator Inequalities, Zagreb, Element (2005).

\bibitem{HPP} F. Hansen, J. Pe\v{c}ari\'{c} and I. Peri\'{c}, Jensen's operator inequality and its converses, Math. Scand., {\bf
100}, 61-73 (2007).

\bibitem{H} J.-B. Hiriart-Urruty and C. Lemarchal, Fundamentals of Convex Analysis, Grundlehren Text Ed., Springer, Berlin (2001).

\bibitem{H2} J.-B. Hiriart-Urruty and J.-E. Mart\'inez-Legaz, Convex solutions of a functional equation arising in information theory, J. Math. Anal. Appl., {\bf 328}, 1309-–1320(2007).

\bibitem{M} P. Mar\'echal, On a functional operation generating convex functions. I. Duality, J. Optim. Theory Appl., {\bf 126}, 175-189(2005).

\bibitem{M2}P. Mar\'echal, On a functional operation generating convex functions. II. Algebraic properties, J. Optim. Theory Appl., {\bf 126}, 357-366(2005).

\bibitem{MA} J.S. Matharu and J.S. Aujla, Some majorization inequalities for convex functions of several variables, Math. Inequal. Appl., {\bf 4}, 947-956(2011).

\bibitem{Mo} B. Mond and J. Pe\v{c}ari\'{c}, Converses of Jensen inequality for several operators, Rev. Anal. Num\'er. Th\'eor. Approx., {\bf 23},  179-–183(1994).

\bibitem{MOS2} M.S. Moslehian, Operator extensions of Hua's inequality, Linear Algebra Appl., {\bf 430}, 1131-1139(2009).

\bibitem{MOS1} M.S. Moslehian and H. Najafi. Around operator monotone functions, Integral Equations Operator Theory, {\bf 71},
575-58(2011).

\bibitem{PET} D. Petz, On the equality in Jensen's inequality for operator convex functions, Integral Equations Operator Theory, {\bf 9}, 744-747(1986).
\end{thebibliography}
\end{document}